\documentclass[11pt]{amsart}
\usepackage{amsthm,amssymb}
\usepackage{verbatim}
\usepackage{hyperref}
\usepackage[dvipsnames]{xcolor}

\usepackage{todonotes,marginnote}  %julien added

\newtheorem{thm}{Theorem}
\newtheorem{thm-known}[subsection]{Theorem}

\newtheorem{lem}[subsection]{Lemma}

\newtheorem{cor}[thm]{Corollary}

\theoremstyle{definition}
\newtheorem{defn}[subsection]{Definition}

\newtheorem{rem}[subsection]{Remark}

\renewcommand{\phi}{\varphi}
\renewcommand{\[}{\begin{equation*}}
\renewcommand{\]}{\end{equation*}}

\newcommand{\Tr}{\mathrm{Tr}}
\newcommand{\tr}{\mathrm{Tr}}
\newcommand{\h}{\mathsf{h}}

\begin{document}
\title[]{On the lower bounds of the $L^2$-norm of the Hermitian scalar curvature }
\author{Julien Keller}
\address{Julien Keller \\ Aix Marseille Universit\'e, CNRS, Centrale Marseille, Institut de Math\'ematiques
de Marseille, UMR 7373, 13453 Marseille, France }
\email{julien.keller@univ-amu.fr}
\author{Mehdi Lejmi}
\address{Mehdi Lejmi \\ Department of Mathematics, Bronx Community College of CUNY, Bronx, NY 10453, USA}
\email{mehdi.lejmi@bcc.cuny.edu}
\begin{abstract}
On a pre-quantized symplectic manifold, we show that the symplectic Futaki invariant, which is an obstruction to
the existence of constant Hermitian scalar curvature almost-K\"ahler metrics, is actually an asymptotic invariant.  This allows us to deduce a lower bound for the $L^2$-norm of the Hermitian scalar curvature
as obtained by S. Donaldson \cite{Don} in the K\"ahler case.

%On a compact symplectic manifold, we prove that the $L^2$-norm of the Hermitian scalar
%curvature, induced by compatible almost-complex structures invariant under a compact subgroup $G$ of Hamiltonian symplectomorphisms,
%is bounded below by a symplectic extension of the Donaldson--Futaki invariant.
%This is done under the hypothesis that the central fiber of the degeneration exists and is smooth
%for any $S^1$-subgroup of $G.$
\end{abstract}

\maketitle
\section{Introduction}

Let $(M,\omega)$ be a symplectic manifold of (real) dimension $2n$.
An almost-complex structure $J$ is $\omega$-compatible
if the tensor $g(\cdot,\cdot)=\omega(\cdot,J\cdot)$
defines a Riemannian metric. The metric $g$ is called then an {{almost-K\"ahler metric}}. When
$J$ is integrable, $g$ is a K\"ahler metric. Given an almost-K\"ahler metric $g,$
one can define the {\it{canonical Hermitian connection}} (see~\cite{gau-1,lib}) $$\nabla_XY=D^g_XY-\frac{1}{2}J(D^g_XJ)Y,$$
where $D^g$ is the Levi-Civita connection of $g$ and $X,Y$ any vector fields on $M.$
The curvature of the induced Hermitian connection on the anti-canonical bundle $\Lambda^{n}(T^{1,0}_JM)$
is of the form $\sqrt{-1}\rho^\nabla.$ The closed (real) $2$-form $\rho^\nabla$ is called
the {\it{Hermitian Ricci form}} and it is a de Rham representative of $2\pi c_1(M, \omega)$ the first Chern class of the tangent bundle $TM$.
The {\it{Hermitian scalar curvature}} $s^\nabla$ of the almost-K\"ahler structure $(\omega,J)$
is then the normalized trace of $\rho^\nabla,$ i.e. 
\begin{equation*}
s^\nabla\omega^n=2n\rho^\nabla\wedge\omega^{n-1}.
\end{equation*}
When the metric is K\"ahler, $s^\nabla$ coincides with the (usual) Riemannian scalar curvature.

We fix now a $2n$-dimensional compact (connected) symplectic manifold $(M,\omega).$ We denote by $AK_\omega$ the (infinite dimensional) Fr\'echet space of all $\omega$-compatible
almost-complex structures and $C_\omega$ the subspace of $\omega$-compatible complex structures. It turns out that
the natural action of the Hamiltonian symplectomorphism group $Ham(M,\omega)$ on $AK_\omega$
is Hamiltonian~\cite{Don-1,Fuj} with moment map $\mu:AK_\omega\longrightarrow\left(Lie(Ham(M,\omega))\right)^\ast$
given by $\mu(J)(f)=\int_M s^\nabla f \frac{\omega}{n!},$ where $s^\nabla$ is the Hermitian scalar curvature of $(\omega,J).$
The induced metrics by the critical points of the functional (defined on $AK_\omega$)
\begin{eqnarray*}
\Vert\mu\Vert^2:J\longmapsto\int_M(s^\nabla)^2\frac{\omega^n}{n!}
\end{eqnarray*}
are called {\it{extremal almost-K\"ahler metrics}}~\cite{Apo-Dra,Lej}. These metrics appear then as a natural
extension of Calabi's extremal K\"ahler metrics~\cite{Cal,Cal-1} to the symplectic setting. The symplectic gradient of the Hermitian
scalar curvature of an extremal almost-K\"ahler metric turns out to be an infinitesimal isometry of the metric.
In particular, constant Hermitian scalar curvature almost-K\"ahler ({\it cHscaK} in short) metrics are extremal.

Furthermore, one can define 
a (geometric) {\it{symplectic Futaki invariant}} (in the K\"ahler case, see~\cite{Fut}). Explicitly,
we fix a compact group $G$ in the Hamiltonian symplectomorphism group $Ham(M,\omega).$
Let $\mathfrak{g}_\omega$ be the space of smooth functions  (with zero integral) which are Hamiltonians with respect to $\omega$
of elements of $\mathfrak{g}=Lie(G).$ Denote by $AK_\omega^G$ (resp. $C_\omega^G$)
the space of all $G$-invariant $\omega$-compatible almost-complex structures (resp  $G$-invariant $\omega$-compatible complex structures). 
Then, we define the map 
\begin{eqnarray*}
\mathcal{F}_\omega^G:\mathfrak{g}&\longrightarrow&\mathbb{R}\\
\mathcal{F}_\omega^G(X)&=&\int_Ms^\nabla \h\,\frac{\omega^n}{n!},
\end{eqnarray*}
where $\h\in\mathfrak{g}_\omega$ is the Hamiltonian induced by $X$
and $s^\nabla$ is the Hermitian scalar curvature induced by any $J\in AK_\omega^G.$
It turns out that $\mathcal{F}_\omega^G$ is independent of the choice of $J\in AK_\omega^G$~\cite{Gau,Lej}.
The map $\mathcal{F}_\omega^G$ is called the symplectic Futaki invariant {{relative to}} $AK_\omega^G$.
It readily follows that if $AK_\omega^G$ contains a cHscaK metric, then $\mathcal{F}_\omega^G\equiv 0.$

In the K\"ahler setting, the {\it{Donaldson--Futaki invariant}} defined in \cite{Don-2}
gives (non-trivial) lower bounds on the Calabi functional \cite{Cal,Cal-1} as proved by S. Donaldson in \cite{Don}.
The existence of constant scalar curvature K\"ahler ({\it cscK} in short) metrics is then
related to an algebro-geometric stability condition, called K-stability, introduced by G. Tian~\cite{Tia} for Fano manifolds (see also~\cite{Din-Tia}).
The {{Donaldson--Futaki invariant}}~\cite{Don,Don-2} is an algebraic invariant which can be defined for singular manifolds
and coincide with the geometric Futaki invariant~\cite{Fut} when the central fiber of the degeneration is smooth. Furthermore, the Donaldson-Futaki invariant has been also defined recently for Sasakian manifolds in \cite{Col-Sze}.

In this paper, we point out that the {{Donaldson--Futaki}} invariant {{may}} be extended
to the symplectic case. Our motivation is that, in the toric case, the existence
of an extremal K\"ahler metric is conjecturally equivalent to the existence of non-integrable extremal
almost-K\"ahler metrics~\cite{Don-2} (see also~\cite[Conjecture 2]{Apo-Cal-Gau-Fri}). Moreover, the examples of toric manifolds studied in~\cite{Don-2} which are not
K-stable do not admit even a cHscaK metric. A related question and also part of the motivation of this work is
the {\it{almost-K\"ahler Calabi-Yau equation}} on $4$-manifolds which has a unique solution if a conjecture of S. Donaldson~\cite{Don-5} holds
(see also~\cite[Question 6.9]{Li} and~\cite{Tos-Wei}). 

More explicitly, let $(M,\omega)$ be a compact symplectic manifold {\it{pre-quantized}} by 
a Hermitian line bundle $(L,h).$ We fix a compact group $G$ in $Ham(M,\omega).$
We consider a $G$-invariant $\omega$-compatible almost-complex structure $J.$
For an integer $k$, we define the {\it{renormalized}} Bochner--Laplacian operator $\Delta_k$ acting on the smooth sections of $L^k$. For a sufficiently large $k>0$, the space $\mathcal{H}_k$ of
the eigensections
of $\Delta_k$, with eigenvalues in some interval depending only on $L,$ is finite dimensional. An orthonormal basis of $\mathcal{H}_k$
gives a `nearly' symplectic and `nearly' holomorphic embedding $\Phi_k : M\longrightarrow \mathbb{P}\mathcal{H}_k^\ast$~\cite{Ma-Mar,Ma-Mar-1},
where the space $\mathbb{P}\mathcal{H}_k^\ast$ can be identified with a $N_k+1$ complex projective space.
Moreover, the line bundles $L^k$ and $\Phi_k^\ast\left(\mathcal{O}(1)\right)$ over $M$ are canonically isomorphic.
The Hermitian metrics $h^k$ on $L^k$ and $h^{\Phi_k^\ast\left(\mathcal{O}(1)\right)}$ (induced by the Hermitian metric on $\mathcal{O}(1)$) on $\Phi_k^\ast\left(\mathcal{O}(1)\right)$
are then related by
\[
h^{\Phi_k^\ast\left(\mathcal{O}(1)\right)}=\frac{h^{k}}{B_k},
\]
%\todo{it is fine if $B_k$ is not defined ?}  -- Yes we should define, I don't know why the end of sentence got commented
where $B_k$ is the {\it{generalized Bergman function}} defined in~(\ref{bergman_kernel}) (see~\cite[Theorem 8.3.11]{Ma-Mar-1}).

Furthermore, the dimension of the space $\mathcal{H}_k$ has an asymptotic expansion of the following type (as consequence of Theorem~\ref{bergman_asym}),
\[\begin{aligned}
   \dim\mathcal{H}_k &=a_0k^{n}+a_1k^{n-1}+O(k^{n-2}),\\
   &= k^n\int_M \frac{\omega^n}{n!} +\frac{ k^{n-1}}{4\pi}\int_Ms^\nabla\frac{\omega^n}{n!}+O(k^{n-2}). 
  \end{aligned}\]
where $s^\nabla$ is the Hermitian scalar curvature of $(\omega,J)$. Observe that the integral $\int_Ms^\nabla\frac{\omega^n}{n!}=\frac{4\pi}{(n-1)!}\int_M c_1(M, \omega)\wedge [\omega]^{n-1}$ is independent of the choice of $J.$

We choose a $S^1$-action $\Gamma$ on $(M,\omega)$ generated by a Hamiltonian vector field in $Lie(G).$ The $S^1$-action on $M$ can be lifted to $L^k$ and induces a linear action $A_k$ on the smooth sections of $L^k.$
Furthermore, this linear action fixes the space $\mathcal{H}_k$ since the $S^1$-action $\Gamma$ preserves the almost-K\"ahler metric induced by $J.$
The trace of this linear action admits an asymptotic expansion (as a consequence of Theorem~\ref{approx})
\[\begin{aligned}
 \mathrm{Tr}(A_k)&=b_0k^{n+1}+b_1k^{n}+O(k^{n-1}),\\
&=-k^{n+1}\int_{M}\h\,\frac{\omega^n}{n!}-\frac{k^n}{4\pi}\int_{M}\h s^\nabla\frac{\omega^n}{n!}+O(k^{n-1}),
  \end{aligned}\]
where the function $\h$ is a Hamiltonian of the $S^1$-action with respect to $\omega.$ We remark that the integral $\int_{M}\h s^\nabla\frac{\omega^n}{n!}$
is independent of the choice of $J\in AK_\omega^G$ the space of all $G$-invariant $\omega$-compatible almost-complex structures~\cite[Lemma 3.1]{Lej}.
%Furthermore, we have that (see Lemma \ref{tr2})\todo{Not sure it is important to say that at that stage, maybe to put somewhere else}
%
%\[\begin{aligned}
% \mathrm{Tr}(A_k^2)=k^{n+2}\int_M \h^2\,\frac{\omega^n}{n!} +O(k^{n+1}).
% \end{aligned}\] 
%

\begin{defn}
The {\it{symplectic Donaldson--Futaki invariant}} $\mathcal{F}^G(\Gamma)$ of the $S^1$-action $\Gamma$ on $(M,L)$
generated by a Hamiltonian vector field in $Lie(G)$ is defined by 
\[
\mathcal{F}^G({\Gamma})=\frac{a_1}{a_0}b_0-b_1.
\]  
\end{defn}

Let $\chi_{\Gamma}:\mathbf{C}^*\hookrightarrow GL(N_{k}+1)$ be a one-parameter
subgroup, such that $\chi_{\Gamma}(S^1)\subset U(N_{k}+1)$ corresponds to the linear action induced by the $S^1$-action $\Gamma$ on $\mathcal{H}_{k}$, i.e. $\chi_{\Gamma}(t)=t^{A_{k}},$ for $t\in S^1.$
Now, we consider the {\it{degeneration}} induced by the family  $ \chi_{\Gamma}(t)\circ \Phi_{k}(M).$  
We suppose that $$M_0=\lim_{t\to 0} \chi_{\Gamma}(t)\circ \Phi_{k}(M)$$ {\it{exists as a symplectic variety with singular locus of complex dimension less than $n-1$}}
(the latter hypothesis is to ensure the existence of the integral~(\ref{eq:lim_ft}), see for instance~\cite{Din-Tia}). 
By definition, $M_0$ is preserved under the action of $\chi_{\Gamma}.$
Then, our main result is that the $L^2$-norm of the zero mean value of the Hermitian scalar curvature of any
$G$-invariant almost-K\"ahler structure whose symplectic form is $\omega$
is bounded below by the {{symplectic Donaldson--Futaki invariant.}}

 \begin{thm}\label{main_thm}
%Let $(M,\omega)$ be a compact symplectic manifold pre-quantized by a Hermitian line bundle $(L,h).$
%Let $\Phi:M\longrightarrow \mathbb{P}\mathcal{H}^\ast\cong\mathbb{CP}^{N}$ be an embedding.
%Let $J$ be an $\omega$-compatible almost-complex inducing a metric $g.$ Suppose there exists a Hamiltonian 
%Killing $S^1$-action on $(M,\omega,g)$

%Let $\chi:\mathbf{C}^*\hookrightarrow GL(N+1)$ be a one-parameter
%subgroup with $\chi(S^1)\subset U(N+1)$ corresponding to the given $S^1$-action.
%Suppose that $M_0=\lim_{t\to 0} \chi(t)\circ \Phi(M)$ {{exists}} and {{smooth}}. Then,
Let $AK_\omega^G$ be the space of all $G$-invariant $\omega$-compatible almost-complex structures. Suppose that $M_0$
exists as a symplectic variety with singular locus of complex dimension less than $n-1$ for all $k$ large and for any $S^1$-subgroup $\Gamma\subset G.$ Then,
\begin{eqnarray*} 
\inf_{{J\in AK_\omega^G }}\Vert s^\nabla - S^\nabla\Vert_{L^2}\geqslant \sup_{\Gamma\subset G}\left(-4\pi\frac{\mathcal{F}^G({\Gamma})}{\Vert\chi_{\Gamma}\Vert}\right),
\end{eqnarray*}
where we denoted 
$s^\nabla$ the Hermitian scalar curvature of $(\omega,J)$ with normalized average $S^\nabla=\frac{\int_Ms^\nabla\omega^n}{\int_M\omega^n}$ and 
$\Vert\chi_{\Gamma}\Vert$ is the leading term of the asymptotic expansion of the norm of the trace-free part $\underline{A_k}$ of $A_k$ i.e.
\begin{equation}\label{ineq0}
\mathrm{Tr}(\underline{A_k^2})=\Vert\chi_{\Gamma}\Vert^2k^{n+2}+O(k^{n+1}).
\end{equation}
The $L^2$-norm $\Vert \cdot\Vert_{L^2}$ is with respect to the volume form $\frac{\omega^n}{n!}$.
%the {{norm}}
%\todo{suggestion: Is it clear that it is a norm from the point of view of algebraic geometry (test-configurations)? It is true from the analytic point of view, using Corollary \ref{lastcor}. It makes sense to define this norm here and relate it to $\tr(A_k^2)$}  of $\chi_{\Gamma}$.
\end{thm}
The asymptotic expansion of $\mathrm{Tr}({A_k^2})$ is computed in Lemma~\ref{tr2} while the expression of $\Vert\chi_{\Gamma}\Vert$ is given by Corollary~\ref{lastcor}.
Our proof of \eqref{ineq0} is direct and differs in part from \cite{Don} (see also the reference \cite{gabor}).

\par Theorem \ref{main_thm} indicates that one can {\it{possibly}} define a notion of {\it{stability}} for the existence of almost-K\"ahler metrics
with constant Hermitian scalar curvature and study the uniqueness of such metrics as done by S. Donaldson in~\cite{Don-3} in the K\"ahler case.
In order to do so, one probably needs to generalize the notion of {{test-configurations}} to the almost-K\"ahler setting
by using symplectic Deligne-Mumford stacks. 

\medskip 

Let us discuss some applications of Theorem \ref{main_thm}. A direct corollary is the following result.
\begin{cor}\label{corollaire}
Suppose that $M_0$ exists as a symplectic variety with singular locus of complex dimension less than $n-1$ for all large $k$ and for any $S^1$-subgroup $\Gamma\subset G.$ 
If an almost-K\"ahler structure $(\omega,J)$ has a constant Hermitian scalar curvature, for any $J\in AK_\omega^G,$ then $\mathcal{F}^G({\Gamma})\geqslant 0$
for any $S^1$-subgroup $\Gamma\subset G.$
\end{cor}
A consequence of Corollary~\ref{corollaire} is that if for a $S^1$-action $\Gamma\subset G$ on a K\"ahler manifold $(M,\omega,J),$ $\mathcal{F}^G({\Gamma})<0$,
then there is no cscK metrics in the K\"ahler class $[\omega]$ on the complex manifold $(M,J)$ since the symplectic Donaldson--Futaki invariant coincides with the Donaldson--Futaki invariant.
Furthermore, we want to stress the fact that there is {\it{no cHscaK metric in $AK_\omega^G$.}} 
In other words, a destabilizing test-configuration in the K\"ahler setting would imply non existence even of cHscaK metrics.
We observe that the K\"ahler metrics in the K\"ahler class $[\omega]$ can be seen as a subspace of $AK_\omega^G$ via Moser's Lemma (see for example~\cite{Duf-Sal}).
If we consider the K-unstable toric examples studied in~\cite{Don-2} for which the destabilizing test configurations satisfy our assumptions, we recover this way the fact that they don't carry cHscaK structures.  We have extra examples of such phenomena for projective bundles.

\begin{cor}\label{corollaire2} %\todo{I think $G=S^1$ in this case no ? Mehdi: Yes, that's possible option}
 Consider $E$ a holomorphic vector bundle over a complex curve of genus $g\geq 2$ of rank $\mathrm{rk}(E)$. Let $\mathbb{P}(E)$ be the complex manifold underlying the total space of the projectivization of $E$. 
 \begin{itemize}
  \item If $\mathrm{rk}(E)= 2$, then the ruled surface $\mathbb{P}(E)$ admits a cHscaK metric if and only if $E$ is polystable.
  \item If $\mathrm{rk}(E)>2$, then the ruled manifold $\mathbb{P}(E)$ admits a cHscaK metric $\omega$ with $C_\omega^{S^1}\neq\emptyset$ if and only if $E$ is polystable. 
 \end{itemize}
\end{cor}

\hfill

{\small
\noindent {\bf Acknowledgments.} 
The authors are very thankful to Wen Lu, Xiaonan Ma and George Marinescu for sharing their paper \cite{Lu-Ma-Mar}.
%in which they prove the {\it{crucial Theorem}} \ref{approx} for this work and also for their very useful comments. 
ML is grateful to Gabor Sz\'ekelyhidi for  very useful discussions and JK thanks Dmitri Panov. Both authors are grateful to Vestislav Apostolov. The work of JK has been carried out in the framework of the Labex Archim\`ede (ANR-11-LABX-0033) and of the A*MIDEX project (ANR-11-IDEX-0001-02), funded by the ``Investissements d'Avenir" French Government programme managed by the French National Research Agency (ANR). JK is also partially supported by supported by the ANR project EMARKS, decision No ANR-14-CE25-0010.
}
%

%{\bf Acknowledgements:} The author is grateful to Gabor Sz\'ekelyhidi for 
%his invaluable help and useful discussions.
%% and for providing him author the preliminary version
%%of his notes `{\it{Introduction to extremal metrics}}' from which this work is inspired.
%The author is very thankful to Wen Lu, Xiaonan Ma and George Marinescu for sending
%the author their paper~\cite{Lu-Ma-Mar} in which they prove the {\it{crucial Theorem}}~\ref{approx} for this work and also for their very useful comments
%on this note. The author would like to thank Vestislav Apostolov for his supporting.

%

\section{Generalized Bergman kernel}

In order to generalize the lower bounds on the Calabi functional as done by S. Donaldson~\cite{Don} to the symplectic case,
we use the eigensections of the {\it {renormalized Laplacian}} operator~\cite{Bor-Uri,Gui-Uri},
defined on smooth sections of a Hermitian line bundle over a compact symplectic manifold, as natural substitutes for the holomorphic sections. Note that we are not working with another natural operator, the spin$^c$ Dirac operator for which other results about Bergman kernel exist, see \cite{DLM}.
%Moreover, Ma--Marinescu~\cite{Ma-Mar,Ma-Mar-1} computed the coefficients of the asymptotic
%expansion of the generalized Bergman kernel associated to this operator. The Hermitian scalar curvature appears then as a coefficient.
%Furthermore, these eigensections yield projective embeddings which are asymptotically symplectic and isometric~\cite{Ma-Mar,Ma-Mar-1}. 

More precisely, let $(M,\omega)$ be a compact symplectic manifold of dimension $2n$.
Suppose that $(M,\omega)$ is {\it{pre-quantized}} by a Hermitian complex line bundle $(L,h)$ which means that the curvature $R^{\nabla^L}$ of some Hermitian connection
$\nabla^L$ of $L$ satisfies $$\frac{\sqrt{-1}}{2\pi}\,R^{\nabla^L}=\omega.$$ This means that
the de Rham class $[{\omega}]$ is integral.

Fix an almost-complex structure $J$ compatible with $\omega$ and denote by $g(\cdot,\cdot)=\omega(\cdot,J\cdot)$ the induced almost-K\"ahler metric. This defines a Laplacian operator $\Delta^{L^k}$ on $L^k$
acting on smooth sections of $L^k$, for $k>0$. Explicitly, $$\Delta^{L^k}=-\sum_{i=1}^{2n}\left(\nabla^{L^k}_{e_i}\right)^2-\nabla^{L^k}_{\left(D^g_{e_i}e_i\right)},$$ where $D^g$
is the Levi-Civita connection with respect to $g$ and $\{e_i\}$ is a local $g$-orthonormal basis of $TM$. The Hermitian metric $h^k$ and connection $\nabla^{L^k}$ on $L^k$
are induced by $h$ and $\nabla^L.$ The renormalized Laplacian is given then by
$$\Delta_k=\Delta^{L^k}-2\pi nk.$$

By a result in~\cite{Ma-Mar-2}, there exists two constants $C_1,C_2>0$ independent of $k$ such that the spectrum of $\Delta_k$ is contained in $(-C_1,C_1)\cup(k\,C_2,+\infty)$ (see also~\cite{Bor-Uri,Gui-Uri}).
Let $\mathcal{H}_k\subset C^\infty(M,L^k)$ be the span of the eigensections of $\Delta_k$ with eigenvalues in $(-C_1,C_1)$.
The space $\mathcal{H}_k$ is then finite dimensional and for large $k$~\cite{Bor-Uri,Gui-Uri,Ma-Mar-2}
\[
\dim \mathcal{H}_k=\int_Me^{k[\omega]}\,Td(T^{1,0}_JM),
\]
where $Td(T^{1,0}_JM)$ is the Todd class of the (complex) vector bundle $T^{1,0}_JM.$

\begin{rem}
When $g$ is K\"ahler and $L$ is a holomorphic Hermitian line bundle, the operator $\Delta_k$ coincides with the $\overline\partial$-Laplacian, 
by the Bochner--Kodaira formula (e.g~\cite[Proposition 3.71]{Ber-Get-Ver}). Then, for large $k$, the space $\mathcal{H}_k$ is exactly the
space of holomorphic sections of $L^k.$
\end{rem} 
 
On sections of $L^k$, we define the inner product
\begin{equation}\label{inner}\langle s_1,s_2\rangle_{L^2}=\int_M(s_1,s_2)_{h^k}\frac{(k\omega)^n}{n!}.\end{equation}

Let $\{s_0,\cdots,s_{N_k}\}$ be an orthonormal basis of $\mathcal{H}_k$. At $x\in M$, the {\it{generalized Bergman function}}  is defined as the restriction to the diagonal of the Bergman kernel, i.e by the formula
\begin{equation}\label{bergman_kernel}
B_k(x)=\sum_{i=0}^{N_k}|s_i(x)|^2_{h^k}.
\end{equation}

X. Ma-- G. Marinescu proved the following asymptotic expansion.
\begin{thm-known}~\cite{Ma-Mar,Ma-Mar-1}\label{bergman_asym}
When $k\rightarrow\infty$,
\begin{equation}\label{expansion}
B_k=1+\frac{s^\nabla}{4\pi}k^{-1}+O(k^{-2}),
\end{equation}
valid in $C^l$ for any $l\geq 0.$ Here, $s^\nabla$ denotes the Hermitian scalar curvature of $(\omega,J)$.
\end{thm-known}

Let $\mathbb{P}\mathcal{H}_k^\ast$ be the projective space associated to the dual of $\mathcal{H}_k$. 
Moreover, once we fix a basis of $\mathcal{H}_k$,
we have an identification $\mathbb{P}\mathcal{H}_k^\ast\cong\mathbb{CP}^{N_k}.$ We have then the following
\begin{thm-known}[\cite{Ma-Mar,Ma-Mar-1}]~
For large $k$, the Kodaira maps $\Phi_k : M\longrightarrow \mathbb{P}\mathcal{H}_k^\ast,$ given by 
\begin{equation*}
\Phi_k(x)=\{s\in\mathcal{H}_k \,|\, s(x)=0\}
\end{equation*}
are well-defined.
\end{thm-known}

Observe that there is a well-defined Fubini-Study form $\omega_{FS}$ on $\mathbb{P}\mathcal{H}_k^\ast$ with a compatible metric $g_{FS}.$ We have then
\begin{thm-known}[\cite{Ma-Mar,Ma-Mar-1}]\label{embedding}~
For large $k,$ we have in $C^\infty$-norm
\begin{align*}
\frac{1}{k}\Phi_k^\ast(\omega_{FS})-\omega=O(k^{-1}),\\
\frac{1}{k}\Phi_k^\ast(g_{FS})-g=O(k^{-1}).
\end{align*}
Moreover, the maps $\Phi_k$ are embeddings and `nearly holomorphic' i.e.
\begin{align*}
\frac{1}{k}\Vert\bar\partial\Phi_k\Vert=O(k^{-1}),\quad\frac{1}{k}\Vert\partial\Phi_k\Vert\geqslant C,\quad {\mathrm{ for\,some }}\,\, C>0.
\end{align*}
\end{thm-known}

Very recently, W. Lu-- X. Ma-- G. Marinescu improved the speed rate of the approximation of the symplectic form. This improvement is actually crucial to obtain the main result of the paper.

\begin{thm-known}[\cite{Lu-Ma-Mar}]~\label{approx}~
For large $k,$ we have in $C^\infty$-norm
\begin{align*}
\frac{1}{k}\Phi_k^\ast(\omega_{FS})-\omega=O(k^{-2}).
\end{align*}
\end{thm-known}

\section{Lower bounds on the $L^2$-norm of the Hermitian scalar curvature}

Let $(M,\omega)$ be a compact symplectic manifold pre-quantized by a Hermitian complex line bundle $(L,h).$
We fix an $\omega$-compatible almost-complex structure $J.$

Given an embedding $\Phi_k:M\longrightarrow \mathbb{P}\mathcal{H}_k^\ast$, for a sufficiently large $k>0$ as in Theorem~\ref{embedding}, we define
a matrix $M(\Phi_k)$ with entries
\[ M(\Phi_k)_{ij} = \int_{M}
\Phi_k^*\left(\frac{Z^i\overline{Z}^j}{|Z|^2}\right)
\frac{(\Phi_k^* \omega_{FS})^{n}}{n!}, \]
where $Z^j$ are homogeneous coordinates on $\mathbb{P}\mathcal{H}_k^\ast.$
Let $\underline{{M}}(\Phi_k)$ denote the trace-free part of $M(\Phi_k).$ 
%and $\omega_{FS}$ is the Fubini-Study metric. Once we fix a basis of $\mathcal{H}_k$,
%we have an identification $\mathbb{P}\mathcal{H}_k^\ast\cong\mathbb{CP}^{N_k}.$

%Let $\underline{M}(\Phi_k)$ denote the trace-free part of $M(\Phi_k)$,
%so that \[ \underline{M}(\Phi_k)_{ij} = M(\Phi_k)_{ij} -
%\frac{k^{n}{\text {Vol}}(M)}{N_k+1}\delta_{ij},\]
%where $N_k+1=\dim\mathcal{H}_k.$

\begin{lem}\label{seq_embed}
Consider $(M,\omega,J)$ a compact almost-K\"ahler manifold pre-quantized by a Hermitian complex line bundle $(L,h)$. Then, there is a sequence of embeddings $\Phi_k:M\longrightarrow
\mathbb{P}\mathcal{H}_k^\ast$ such that
\[ \Vert \underline{{M}}(\Phi_k)\Vert \leqslant \frac{k^{n/2-1}}{4\pi}\Vert
  s^\nabla - S^\nabla\Vert_{L^2} + O(k^{n/2-2}). \]
  Here $\Vert \underline{M}(\Phi_k)\Vert = \left(\mathrm{Tr}\left(\underline{M}(\Phi_k)\right)^2\right)^{1/2},$ $s^\nabla$ is the Hermitian scalar curvature of $(\omega, J)$
  and $ S^\nabla=\frac{\int_Ms^\nabla\omega^n}{\int_M\omega^n}$ is the normalized average of $s^\nabla.$
\end{lem}
\begin{proof}
This is done as in the K\"ahler case. For the reader's convenience, we reproduce here the proof. 
We use the sequence of embeddings $\Phi_k$ defined by the orthonormal bases $\{s_0,\cdots,s_{N_k}\}$ of
$\mathcal{H}_k$. Using Theorem~\ref{approx}, we have that
  
\[ \begin{aligned}
         M(\Phi_k)_{ij} &= \int_{M}
         \Phi_k^*\left(\frac{Z^i\overline{Z}^j}{|Z|^2}\right)
         \ \frac{(\Phi_k^*\omega_{FS})^{n}}{n!} ,\\
         &= \int_M \frac{( s_i, s_j)_{h^k}}{B_k}
          \frac{ (k\omega)^{n}\left(1 + O(k^{-2})\right)}{n!}.
       \end{aligned} \]
We can assume that $M$ is
   diagonal. Then, using Theorem~\ref{bergman_asym}, we obtain 
\begin{equation}   \begin{aligned}\label{trace_part}
     M(\Phi_k)_{ii} &= k^{n}\int_M \frac{|s_i|^2_{h^k}}{B_k} 
      \frac{\omega^{n}\left(1 + O(k^{-2})\right)}{n!} , \\
     &= k^{n}\int_M |s_i|^2_{h^k} \left(1-\frac{s^\nabla}{4\pi}k^{-1}\right)
      \frac{\omega^{n}}{n!} + O(k^{n-2}),\\
      &=1-\frac{k^{-1}}{4\pi}\int_M |s_i|^2_{h^k}\,\, s^\nabla
      \frac{(k\omega)^{n}}{n!} + O(k^{n-2}).
   \end{aligned}\end{equation}
From Theorem~\ref{bergman_asym}, the dimension of $\mathcal{H}_k$ is given by
   \begin{eqnarray}\begin{aligned}\label{dim_exp}
   N_k + 1 %&=a_0k^{n}+a_1k^{n-1}+O(k^{n-2}),\\
   &= k^n\int_M \frac{\omega^n}{n!} +\frac{ k^{n-1}}{4\pi}\int_Ms^\nabla\frac{\omega^n}{n!}+O(k^{n-2}). 
   \end{aligned}\end{eqnarray}
   It follows that
   \[ \begin{aligned}\sum_{i=0}^{N_k} M(\Phi_k)_{ii} &= N_k+1-\frac{k^{-1}}{4\pi}\int_M B_{k}\,\, s^\nabla
      \frac{(k\omega)^{n}}{n!} + O(k^{n-2}),\\
    &= N_k+1-\frac{k^{-1}}{4\pi}\int_M \,\, s^\nabla
      \frac{(k\omega)^{n}}{n!} + O(k^{n-2}),
  \end{aligned} \]
  
  Hence
  \[ \begin{aligned} \frac{\mathrm{Tr}(M(\Phi_k))}{N_k+1}=
  1-\frac{k^{-1}}{4\pi}{S^\nabla}+O(k^{-2}).
  \end{aligned} \]
  
  Combined with~(\ref{trace_part}), the trace free part $\underline{M}(\Phi_k)$ of $M(\Phi_k)$ is
   \[ \begin{aligned}\underline{M}(\Phi_k)_{ii}=-\frac{k^{-1}}{4\pi}\int_M 
   |s_i|^2_{h^k}(s^\nabla-S^\nabla)\frac{(k\omega)^n}{n!}+O(k^{-2}).
   \end{aligned} \]
  By the Cauchy--Schwarz inequality, we have
   \[ \begin{aligned}|\underline{M}(\Phi_k)_{ii}|^2&\leqslant\frac{k^{-2}}{16\pi^2}\int_M 
   |s_i|^2_{h^k}\frac{(k\omega)^n}{n!}\int_M|s_i|^2_{h^k}(s^\nabla-S^\nabla)^2\frac{(k\omega)^n}{n!}+O(k^{-3}),\\
   &=\frac{k^{-2}}{16\pi^2}\int_M|s_i|^2_{h^k}(s^\nabla-S^\nabla)^2\frac{(k\omega)^n}{n!}+O(k^{-3}).
   \end{aligned} \]
  Taking the sum, we obtain that
   \[ \begin{aligned}\Vert\underline{M}(\Phi_k)\Vert^2&\leqslant\frac{k^{-2}}{16\pi^2} 
   \int_MB_k(s^\nabla-S^\nabla)^2\frac{(k\omega)^n}{n!}+O(k^{n-3}),\\
   &=\frac{k^{n-2}}{16\pi^2}\int_M(s^\nabla-S^\nabla)^2\frac{\omega^n}{n!}+O(k^{n-3}).
   \end{aligned} \]
  The Lemma follows.

\end{proof}

Our aim now is to find a lower bound for $\Vert\underline{M}(\Phi_k)\Vert.$
First, we fix a compact group $G$ in $Ham(M,\omega)$. 
We consider a $G$-invariant $\omega$-compatible almost-complex structure $J.$
We choose a $S^1$-action $\Gamma$ on $(M,\omega)$ generated by a Hamiltonian vector field in $Lie(G).$ 
The $S^1$-action can be lifted to an action on $L^k$ (preserving $h^k$ and $\nabla^{L^k}$) (for any $k\geqslant 1$).
This induces a linear action of $S^1$ on smooth sections of $L^k$. Furthermore, since the $S^1$-action preserves the induced metric by $J$, the induced action maps
$\mathcal{H}_k$ to itself. We denote by $-\sqrt{-1}A_k$ the
 infinitesimal generator of the linearized $S^1$-action $\Gamma$ on $\mathcal{H}_k$ with $A_k$ having integral entries.

For large $k>0$, let $\Phi_{k}:M\longrightarrow
\mathbb{P}\mathcal{H}_k^\ast$ be an embedding of $M$ using an orthonormal bases $\{s_0,\cdots,s_{N_k}\}$ of $\mathcal{H}_k$.
Let %\todo{Isn't it better to use $\chi_\Gamma$ everywhere as in the introduction ?}
$\chi_{\Gamma}:\mathbf{C}^*\hookrightarrow GL(N_k+1)$ be a one-parameter
subgroup, such that $\chi_{\Gamma}(S^1)\subset U(N_k+1)$ satisfying $\chi_{\Gamma}(t)=t^{A_k}$ (normalized so that $\chi_{\Gamma}(1)$ is the identity map).
By definition, $\chi_{\Gamma}(S^1)$ preserves both the Fubini-Study form $\omega_{FS}$ and $g_{FS}$ on $\mathbb{P}\mathcal{H}_k^\ast.$
A Hamiltonian function (with respect to $\omega_{FS}$) for the corresponding $S^1$-action
is given by
\[
\h_{A_k} = \frac{-\sum_{i,j}\left(A_k\right)_{ij}Z^i\overline{Z}^j}{|Z|^2}.
 \]
 so that
 \begin{equation}\label{Phikast}
\Phi_k^\ast(\h_{A_k}) = \frac{-\sum_{i,j}\left(A_k\right)_{ij}(s_i,s_j)_{h^k}}{B_k}.
\end{equation}

%We fix an embedding $\Phi_{k}:M\longrightarrow
%\mathbb{P}\mathcal{H}_k^\ast$ in the sequence satisfying Lemma~\ref{seq_embed}.
%We have that $L^{k}$ is canonically isomorphic to $\Phi_k^\ast(\mathcal{O}(1)),$ where $\mathcal{O}(1)$
%is the hyperplane bundle over $\mathbb{P}\mathcal{H}_k^\ast$~\cite[Theorem 8.3.11]{Ma-Mar-1}.
%Let $\chi:\mathbf{C}^*\hookrightarrow GL(N_k+1)$ be a one-parameter
%subgroup, such that $\chi(S^1)\subset U(N_k+1)$ is a (smooth) Killing Hamiltonian action preserving both the Fubini-Study form $\omega_{FS}$ and $g_{FS}$ on $\mathbb{P}\mathcal{H}_k^\ast.$
%So, $\chi(t)=t^A$ for a Hermitian matrix $A$ with integer
%eigenvalues. A Hamiltonian function for the corresponding $S^1$-action
%is given by
%\[ h_{A} = \frac{A_{ij}Z^i\overline{Z}^j}{|Z|^2}. \]

%

Now, let $\Phi_k^t = \chi_{\Gamma}(t)\circ \Phi_k$ and define the function
\[ f(t) =- \mathrm{Tr}(\underline{{A_k}{M}}(\Phi_k^t)) =-
\mathrm{Tr}\left(\underline{A_k}M(\Phi_k^t)\right), \]
where $\underline{A_k}$ is the trace-free part of $A_k$. Then
\[ \label{eq:ft}
f(t) = \int_M \Phi_k^{t*}(\h_{A_k})\,\frac{
  (\Phi_k^{t*}\omega_{FS})^{n}}{n!} +
\frac{\mathrm{Tr}(A_k)}{N_k+1} \int_M \frac{
(\Phi_k^{t*}\omega_{FS})^{n}}{n!}. \]
A calculation shows that for real numbers $t>0$ we
have $f'(t)\geqslant 0$.
\begin{lem}
  With the above definition, one has $\forall t>0$, $$f'(t)\geqslant 0.$$
\end{lem}
\begin{proof}
  We consider the one-parameter
group of diffeomorphisms generated by the vector field
$-{\mathrm{grad}}\,\h_{A_k}$ so we are
approaching $0$ along the positive real axis in $\mathbb{C}^\ast$.
Then, we have the following derivative at $s=0$
\begin{equation}\label{compute_ft}\begin{aligned}
\left. \frac{d}{ds}\right |_{s=0}\int_{M}  \Phi_k^{s*}(\h_{A_k})\,\frac{
  (\Phi_k^{s*}\omega_{FS})^{n}}{n!}=&-\int_{\Phi_k(M)} |{\mathrm{grad}}\,\h_{A_k}|^2\,\frac{
  \omega_{FS}^{n}}{n!}\\
  &+\int_{\Phi_k(M)} \h_{A_k}\,\frac{
  \mathfrak{L}_{-{\mathrm{grad}}\,\h_{A_k}}\omega_{FS}\wedge\omega_{FS}^{n-1}}{(n-1)!}.
\end{aligned}
\end{equation}
The second term in the r.h.s of (\ref{compute_ft}) can be written as
\[ \begin{aligned}
\int_{\Phi_k(M)} \h_{A_k}\,
  \mathfrak{L}_{-{\mathrm{grad}}\,\h_{A_k}}\omega_{FS}\wedge\omega_{FS}^{n-1}&=-  \int_{\Phi_k(M)} d \h_{A_k}\wedge d^c \h_{A_k}
 \wedge\omega_{FS}^{n-1},\\
 &=\frac{1}{n}\int_{\Phi_k(M)} |{}d \h_{A_k}|_M^2\,\omega_{FS}^{n},
\end{aligned}
\]
where $ |{}d \h_{A_k}|_M^2=|{\mathrm{grad}}\,\h_{A_k}|_M^2$ is the norm of the tangential part to $\Phi_k(M)$. We deduce
\[ \begin{aligned}
\left. \frac{d}{ds}\right |_{s=0}\int_{M}  \Phi_k^{s*}(\h_{A_k})\,\frac{
  (\Phi_k^{s*}\omega_{FS})^{n}}{n!}&=-\int_{\Phi_k(M)} |{\mathrm{grad}}\,\h_{A_k}|_N^2\,\frac{
  \omega_{FS}^{n}}{n!},
\end{aligned}
\] 
where $|{\mathrm{grad}}\,\h_{A_k}|_N^2$ is the norm of the normal component. 
On the other hand
\[ \begin{aligned}
\left. \frac{d}{ds}\right |_{s=0}\int_{M}  \frac{
  (\Phi_k^{s*}\omega_{FS})^{n}}{n!}&=0
  \end{aligned}
\]  

Increasing $t$
corresponds to flowing along ${\mathrm{grad}}\,\h_{A_k}.$ We deduce that $f'(t)\geqslant 0$
for real numbers $t>0$.
\end{proof}
 
Now it follows that
\[ -\mathrm{Tr}(\underline{A_k\,M}(\Phi_k)) = f(1) \geqslant
\lim_{t\to 0} f(t), \]
and so by the Cauchy--Schwarz inequality
\begin{equation} \label{cauchy_schwarz}\Vert\underline{A_k}\Vert\, \Vert\underline{M}(\Phi_k)\Vert \geqslant
\lim_{t\to 0}f(t). \end{equation}
In particular if $\lim_{t\to 0} f(t) > 0$, then we get a positive lower bound
on $\Vert \underline{M}(\Phi_k)\Vert$.
%We need to compute the limit on the right hand side.

\hfill

Suppose now that {\it{the limit $M_0=\lim_{t\to 0} \Phi_k^t(M)$ exists as a symplectic variety with singular locus of complex dimension less than $n-1$}}.
%By definition, $M_0$ is preserved by the action of the one-parameter subgroup $\chi(t)$. Moreover,
%since the action of $\chi(S^1)$ is Hamiltonian with respect to $\omega_{FS}$, it can be lifted to any Hermitian complex line bundle
%polarizing $M_0$, namely $L_0^m$ (for $m>k$) the $m$-tensor power of the restriction of the $\mathcal{O}(1)$ to $M_0.$
We have then
\begin{equation}\label{eq:lim_ft}
\lim_{t\to 0}f(t) = \int_{M_0} \h_{A_k}\,\frac{
  \omega_{FS}^{n}}{n!} +
\frac{\mathrm{Tr}(A_k)}{N_k+1} \int_{M_0} \frac{
 \omega_{FS}^{n}}{n!}.\end{equation}
  
%So, we compute~(\ref{eq:lim_ft}) in the case when $M_0$ exists and smooth.
It follows from Theorem~\ref{approx} that  one can choose a Hamiltonian $\h$ with respect to
$\omega$ such that 
\begin{equation}\label{approx_sympl}
\frac{1}{k}\Phi^\ast_k \left(\h_{A_k}\right)-\h=O(k^{-2}).
\end{equation}
Then
\[\begin{aligned}
\int_{M}hB_k\frac{\omega^n}{n!}&=\frac{1}{k}\int_{M}\Phi_k^\ast(\h_{A_k})B_k\frac{\omega^n}{n!}+O(k^{-2}),\\
&=-\frac{1}{k^{n+1}}\sum_{i,j}{\left(A_k\right)_{ij}}\int_{M}(s_i,s_j)_{h^k}\frac{(k\omega)^n}{n!}+O(k^{-2}),\\
&=-\frac{1}{k^{n+1}}\mathrm{Tr}(A_k)+O(k^{-2}).
\end{aligned}
\] 
It follows from Theorem~\ref{bergman_asym} that 
\begin{eqnarray}\begin{aligned}\label{A_exp} 
\mathrm{Tr}(A_k)&=b_0k^{n+1}+b_1k^n+O(k^{n-1}),\\
&=-k^{n+1}\int_{M}\h\,\frac{\omega^n}{n!}-\frac{k^n}{4\pi}\int_{M}\h s^\nabla\frac{\omega^n}{n!}+O(k^{n-1}).
\end{aligned}
\end{eqnarray} 
One has also from~(\ref{approx_sympl})
\begin{equation}\label{first_term}
\int_{M_0} \h_{A_k}\,\frac{\omega_{FS}^{n}}{n!}= -b_0k^{n+1}+O(k^{n-1}).
 \end{equation}

Then, from~(\ref{cauchy_schwarz}),~(\ref{eq:lim_ft}),~(\ref{dim_exp}),~(\ref{A_exp}) and~(\ref{first_term}), we deduce
\[ \begin{aligned}\Vert\underline{A_k}\Vert\, \Vert\underline{M}(\Phi_k)\Vert \geqslant &-b_0k^{n+1}+\frac{b_0k^{n+1}+b_1k^n+O(k^{n-1})}{a_0k^{n}+a_1k^{n-1}+O(k^{n-2})}a_0k^n+O(k^{n-1}),\\
=&-b_0k^{n+1}\\
&+\left(b_0k^{n+1}+b_1k^n+O(k^{n-1})\right)\left(1-\frac{a_1}{a_0}k^{-1}+O(k^{-2})\right)\\
&+O(k^{n-1}),\\
=&k^n\left(b_1-\frac{a_1}{a_0}b_0\right)+O(k^{n-1}).
\end{aligned}\]

It follows then from Lemma~\ref{seq_embed} that
\begin{align}\label{final_ineq}\Vert\underline{A_k}\Vert\left(\frac{k^{n/2-1}}{4\pi}\Vert
  s^\nabla - S^\nabla\Vert_{L^2} + O(k^{n/2-2})\right)\geqslant &   k^n\left(b_1-\frac{a_1}{a_0}b_0\right)\\
  &+O(k^{n-1}).\nonumber
  \end{align}
% where $\Vert\chi_{\Gamma}\Vert$ is the first term in 
% \begin{equation}\label{norm_asym}
% \Vert\underline{A_k}\Vert^2=\Vert\chi_\Gamma\Vert^2k^{n+2}+O(k^{n+1}).
% \end{equation}
% \begin{defn}
%The {\it{symplectic Donaldson--Futaki invariant}} is defined by $$F(\chi_G)=\frac{a_1}{a_0}b_0-b_1,$$
%where $a_0,a_1,b_0,b_1$ are defined in~(\ref{dim_exp}) and~(\ref{A_exp}).
%An alternative definition is given by $$F(\chi)=\frac{1}{4\pi}\int_{M}\left(s^\nabla-S^\nabla\right)h\frac{\omega^n}{n!},$$
%where $s^\nabla$ is the Hermitian scalar curvature of $(\omega,J),S^\nabla$ its normalized average and $h$ is the Hamiltonian corresponding
%to the $S^1$-action.
%\end{defn}

\bigskip
%\color{blue}\todo{remove blue once checked}
Now, we need to compute the asymptotic expansion for $\Vert{A_k}\Vert^2=\Tr(A_k^2).$ Let us denote $\nu=\omega^n/n!$ and consider $P_{\nu,k}$ the smooth kernel of the $L^2$-orthogonal projection from $C^\infty(M,L^k)$ to $\mathcal{H}_k$. Set $$K_k(x,y)=\vert P_{\nu,k}(x,y)\vert^2_{h^k\otimes (h^k)^*},$$
where $x,y\in M$. We can write
$$K_{k}(x,y)={k^{n}}\sum_{i,j=1} (s_i(x),s_j(x))_{h^k}(s_j(y),s_i(y) )_{h^k},$$
for $\{s_i\}$ an $L^2$-orthonormal basis with respect to the inner product \eqref{inner}. 
We consider the integral operator associated to $K_k$ which is defined for any $f\in C^{\infty}(M)$ as 
$$Q_{K_k}(f)(x)=\int_{X}K_{k}(x,y)f(y) \frac{\omega^n_y}{n!}.$$
The $Q$-operator has been studied by S. Donaldson \cite{Don-4}, K. Liu-- X. Ma \cite{LM} and X. Ma-- G. Marinescu \cite{Ma-Mar3} in the context of  K\"ahler compact manifolds. They provided an asymptotic result for this operator. We quote a generalization of this result obtained by W. Lu-- X. Ma-- G. Marinescu to the context of pre-quantized symplectic compact manifolds.

\begin{thm-known}[\cite{Lu-Ma-Mar2}]\label{Qkthm}
For any integer $m\geq 0$, there exists a constant $c>0$ such that for any $f\in C^{\infty}(M)$, 
\begin{align}\label{asymptQ2}\Big\Vert Q_{K_k}(f) - f\Big\Vert_{C^m} \leq \frac{c}{k}\Vert f \Vert_{C^{m+2}}.\end{align}
Moreover, \eqref{asymptQ2} is uniform in the sense that there is an integer $s_0$ such that if the hermitian metric $h$ on $L$ varies in a bounded set in $C^{s_0}$ topology then the constant $c$ is independent of $h$.
\end{thm-known}

\begin{lem}\label{tr2}
 With notations as above, $$\Tr(A_k^2)=k^{n+2}\int_M \h^2\,  \frac{\omega^n}{n!}+O(k^{n+1}),$$
 where $\h$ is  a hamiltonian defined by $\omega$.
\end{lem}

\begin{proof}
 Let us write
  \begin{equation}\label{tA}\tilde{A}_{ij}=k^n\int_M (s_i, \Phi^*_k(\h_{A_k}) s_j )_{h^k} \frac{\omega^n}{n!},
   \end{equation}
 where $\Phi^*_k(\h_{A_k})$ is given by \eqref{Phikast} and $\{s_i\}$ is a fixed $L^2$-orthonormal basis of holomorphic sections with respect to the inner product \eqref{inner}.
Now, set
 \begin{align*}Q(A_k)_{ij}& ={k^{n}}\int_M (s_i,  \sum_{p,q} (A_k)_{pq} (s_p,s_q)_{h^k}   s_j )_{h^k} \frac{\omega^n}{n!},\\
  & ={k^{n}} \int_M   \sum_{p,q} (A_k)_{pq} (s_p,s_q)_{h^k}(s_i,  s_j)_{h^k} \frac{\omega^n}{n!}.
   \end{align*}
With the map $\iota:Met(\mathcal{H}_k)\rightarrow C^\infty(M)$ given by
$$\iota(A_{ij})=\sum_{i,j} A_{ij}(s_i,s_j)_{h^k},$$
one can write $\iota \circ Q(A_k)=Q_{K_k}\circ \iota(A_k)$. The map $\iota$ is linear and invertible on its image. From Theorem \ref{Qkthm}, we have 
\begin{equation} Q(A_k)=A_k(Id+O(1/k)).\label{eqQA}\end{equation}
The Bergman function has a uniform asymptotic expansion as stated in Theorem \ref{bergman_asym}. From  the higher order term of this expansion, we can deduce using \eqref{tA}, \eqref{Phikast} and \eqref{eqQA}   
that $$\tilde{A}_{ij}=-Q(A_k)(Id+O(1/k))=-A_k(Id+O(1/k)).$$
Consequently, 
$$\tr(A_k^2)=\tr(\tilde{A}^2)(1+O(1/k)).$$
Now, let us compute $\tr(\tilde{A}^2)$. By a direct computation, we have 
\begin{align*}\tr(&\tilde{A}^2) \\
    &= k^{2n}\int_{M \times M}\sum_{i,j} \left(s_{i}(x),\Phi^*_k(\h_{A_k})(x)s_{j}(x) \right) \left(s_{j}(y),\Phi^*_k(\h_{A_k})(y)s_{i}(y) \right) \omega_x^n\omega_y^n, \\
    &= k^{n}\int_M \tr(Q_{K_k}(\Phi^*_k(\h_{A_k}))\Phi^*_k(\h_{A_k})) \omega^n, \\
            &= k^{n+2}\int_M \tr\left(Q_{K_k}\left(\frac{1}{k}\Phi^*_k(\h_{A_k})\right)\frac{1}{k}\Phi^*_k(\h_{A_k})\right) \omega^n.
            \end{align*}
We have $Q(\frac{1}{k}\Phi^*_k(\h_{A_k})) = \frac{1}{k}\Phi^*_k(\h_{A_k})(1+ O(\frac{1}{k}))$ from Theorem \ref{Qkthm} and also $\frac{1}{k}\Phi^*_k(\h_{A_k})=\h(1+O(\frac{1}{k}))$ from \eqref{approx_sympl}. Combining all previous results, we obtain the asymptotic of $\Tr(A_k^2)$. 
\end{proof}
Let us write $\mathrm{Tr}(\underline{A_k^2})$ as
$$
\mathrm{Tr}(\underline{A_k^2})=\Vert\chi_{\Gamma}\Vert^2k^{n+2}+O(k^{n+1}).$$	
Then, the expression of $\Vert\chi_{\Gamma}\Vert$ is given by the following result.
\begin{cor}\label{lastcor}
 With notations as above $$\Vert \chi_\Gamma\Vert^2=\int_M (\h-\widehat{\h})^2\frac{\omega^n }{n!},$$
 with $\widehat{\h}$ the normalized average of $\h$.
\end{cor}

\begin{proof}[Proof of Theorem \ref{main_thm}]
The proof is now obtained by combining Lemma~\ref{tr2} and ~(\ref{final_ineq}) and letting $k\to\infty$.
\end{proof}

\begin{proof}[Proof of Corollary \ref{corollaire2}]
 We know from Narasimhan and Seshadri that if $E$ is polystable then $\mathbb{P}(E)$ admits a cscK metric (in any K\"ahler class) and thus a cHscaK metric, see \cite{Apo-Cal-Gau-Fri} for details. Now, assume that we have a symplectic form such that $C^{S^1}_\omega\neq\emptyset $ i.e there is an ${S^1}$-invariant integrable compatible almost-complex structure $J$. If $E=E_1\oplus ...\oplus E_s$ is not polystable and $F$ is a destabilizing subbundle of one component of $E$, say $E_1$, one can consider the test configuration  associated to the deformation to the normal cone of $\mathbb{P}(F\oplus E_2...\oplus E_s)$ whose central fibre is $\mathbb{P}(F \oplus E_1/F\oplus E_2\oplus ..\oplus E_s)$ and in particular is smooth. This test configuration admits a $\mathbb{C}^*$ action that covers the usual action on the base $\mathbb{C}$ and whose restriction to $F\oplus E_1/F\oplus E_2\oplus ..\oplus E_s$ scales the fibers of $F$ with weight 1 and acts trivially on the other components. Seeing $(\mathbb{P}(E),\omega,J)$ as a K\"ahler manifold, the computations of \cite[Section 5]{Ross-Thomas} (see also \cite{DVZ}) show that the Futaki invariant of this test configuration is negative. Actually, the Futaki invariant is a positive multiple of the difference of the slopes $\mu(E_1)-\mu(F)<0$. Then, we apply  Corollary \ref{corollaire} to deduce the non existence of cHscaK structure in $ AK_\omega^{S^1}$.  In the case of $\mathrm{rk}(E)= 2$, any symplectic rational ruled surface admits a compatible integrable complex structure, see \cite{AGK} and references therein. Note that for the general case, it is unclear whether we can drop the assumption on $C_\omega$ as there exist projective manifolds with symplectic forms $\omega$ such that $C_\omega=\emptyset$, see for instance \cite{CP}.
\end{proof}

\providecommand{\bysame}{\leavevmode\hbox to3em{\hrulefill}\thinspace}
\providecommand{\MR}{\relax\ifhmode\unskip\space\fi MR }
% \MRhref is called by the amsart/book/proc definition of \MR.
\providecommand{\MRhref}[2]{%
  \href{http://www.ams.org/mathscinet-getitem?mr=#1}{#2}
}
\providecommand{\href}[2]{#2}

\end{document}